\documentclass[12pt,reqno,tbtags]{amsart}

\usepackage[]{amsmath,amssymb,amsfonts,latexsym,amsthm,enumerate,color}
\usepackage{amssymb}
\usepackage[numeric,initials,nobysame]{amsrefs}

\numberwithin{equation}{section}

\newtheorem{theorem}{Theorem}[section]
\newtheorem*{theorem*}{Theorem}
\newtheorem{lemma}[theorem]{Lemma}
\newtheorem{claim}[theorem]{Claim}

\newtheorem*{observation*}{Observation}

\theoremstyle{definition}{

\newtheorem*{remark*}{Remark}
}

\newcommand{\R}{\mathbb R}

\newcommand{\Q}{\mathbb Q}
\newcommand{\N}{\mathbb N}
\newcommand{\Z}{\mathbb Z}

\newcommand{\eps}{\varepsilon}
\newcommand{\ra}{\rightarrow}

\newcommand{\lam}{\lambda}

\newcommand{\E}{\mathbb{E}}
\renewcommand{\P}{\mathbb{P}}

\newcommand{\cF}{\mathcal{F}}

\renewcommand{\epsilon}{\varepsilon}

\newcommand{\M}{\mathbb M}
\DeclareMathOperator{\sgn}{sgn}

\newcommand{\LL}{\mathcal{L}}

\date{}

\begin{document}
\title{Poisson Thickening}

\author{Ori Gurel-Gurevich}
\address{Ori Gurel-Gurevich\hfill\break
University of British Columbia\\
1984 Mathematics Road\\
Vancouver, BC, V6T 1Z2, Canada.}
\email{origurel@math.ubc.ca}
\urladdr{http://www.math.ubc.ca/~origurel/}

\author{Ron Peled}
\thanks{Research of R.P. supported by NSF Grant OISE 0730136.}
\address{Ron Peled\hfill\break
Tel Aviv University\\
School of Mathematical Sciences\\
Ramat Aviv, Tel Aviv, 69978, Israel.}
\email{peledron@post.tau.ac.il}
\urladdr{http://www.math.tau.ac.il/~peledron}

\begin{abstract}
Let $X$ be a Poisson point process of intensity $\lambda$ on the real line. A \emph{thickening} of it is a (deterministic) measurable function $f$ such that $X \cup f(X)$ is a Poisson point process of intensity $\lambda'$ where $\lambda'>\lambda$. An \emph{equivariant} thickening is a thickening which commutes with all shifts of the line. We show that a thickening exists but an equivariant thickening does not. We prove similar results for thickenings which commute only with integer shifts and in the discrete and multi-dimensional settings. This answers 3 questions of Holroyd, Lyons and Soo.

We briefly consider also a much more general setup in which we ask for the existence of a deterministic coupling satisfying a relation between two probability measures. We present a conjectured sufficient condition for the existence of such couplings.
\end{abstract}

\maketitle

\section{Introduction and Results}

\subsection{Main Theorems}Let $\M$ be the space of locally finite sets in $\R$, endowed with its standard $\sigma$-algebra\footnote{This is the minimal $\sigma$-algebra under which all projection maps $\mu_B$ are measurable, for any borel $B\subset \R$, where $\mu_B(S)$ is the cardinality of $S\cap B$.}.
We view Poisson processes on $\R$ as random elements of $\M$. For $\lam'>\lam>0$, we call a measurable function $f:\M\ra\M$ a \emph{thickening} (from intensity $\lam$ to $\lam'$) if $X\cup f(X)$ is a Poisson process of intensity $\lam'$ when $X$ is a Poisson process of intensity $\lam$. Thus, $f$ adds points to $X$, at locations which are determined solely by $X$, and produces a Poisson process of higher intensity.
A thickening $f$ is \emph{equivariant} if $\sigma \circ f = f \circ \sigma$ for any shift operator $\sigma:\R\to\R$. The following theorems address the existence of thickenings and equivariant thickenings.

\begin{theorem} \label{no}
An equivariant thickening does not exist for any $\lambda'>\lambda>0$.
\end{theorem}

\begin{theorem} \label{yes}
A (\emph{non}-equivariant) thickening exists for every $\lambda'>\lambda>0$.
\end{theorem}

We remark that in other works on equivariant extension of processes, it is also common to have the function $f$ depend on additional randomness and have the equality $\sigma \circ f = f \circ \sigma$ hold only in distribution. In our context, existence of these so called randomized equivariant thickenings is trivial.

We turn now to discrete analogues of the above theorems.
For $0<p<p'<1$, let $X= \{X_i\}_{i\in\Z}$ be a sequence of i.i.d. $\{0,1\}$-valued random variables with $\E(X_0)=p$. A measurable function $f:\{0,1\}^\Z \ra \{0,1\}^\Z$ is called a \emph{discrete thickening} (from density $p$ to density $p'$), if the sequence $\{f(X)_i\}_{i\in\Z}$ is i.i.d. with $\E(f(X)_i)=p'$ and $f(X)_i \ge X_i$ for all $i\in\Z$. $f$ is called \emph{equivariant} if $\sigma \circ f = f \circ \sigma$ where $\sigma:\{0,1\}^\Z \ra \{0,1\}^\Z$ is defined by $\sigma(x)_i=x_{i+1}$. An equivariant function cannot increase the entropy of a process (a consequence of Kolmogorov-Sinai theorem, see e.g. \cite{P}, Chapter 5); hence, there is no equivariant discrete thickening for $p < p'< 1-p$. In \cite{B2}, Ball showed that equivariant discrete thickenings \emph{do exist} when $1-p<p'$. The case $p'=1-p$ (for $p<\frac{1}{2}$) appears to have not been treated.
The next theorem shows that as far as \emph{non}-equivariant thickenings are concerned, the discrete and continuous cases are similar.

\begin{theorem} \label{discrete}
A (\emph{non}-equivariant) discrete thickening exists for every $0<p<p'<1$.
\end{theorem}

These three theorems answer the three questions posed by Holroyd, Lyons and Soo in \cite{HLS}, where the related problem of splitting a Poisson process was addressed. They also showed that no \emph{strongly finitary} thickening exists (equivariant or not).
The problem of thinning a Poisson process was considered earlier by Ball \cite{B1} and also expanded upon in \cite{HLS}. We refer the reader to \cite{HLS} for more background.

\subsection{Higher dimensions}
A Poisson process is well-defined over any $\sigma$-finite measure space and the definition of thickening extends naturally to this case. We note that if $\mathcal{S}$ is a measure space which is isomorphic to $\R$ then, since Poisson processes of intensity $\lambda$ on $\R$ pass to Poisson processes of intensity $\lambda$ on $\mathcal{S}$ via this isomorphism and vice versa, Theorem~\ref{yes} implies that a (\emph{non}-equivariant) thickening also exists for Poisson processes on $\mathcal{S}$ for every $\lambda'>\lambda>0$.

One may generalize Theorems~\ref{no} and \ref{yes} in two ways. First, one may consider a weaker equivariance condition by only requiring that the thickening $f$ commute with shifts $\sigma$ taken from some sub-group of all shifts. Second, one may consider a multi-dimensional setting in which $\R$ is replaced by $\R^d$. The following theorem extends our results to this more general scenario.

\begin{theorem}\label{cont_gen}
For any $d\ge 1$ and any pair of intensities $\lambda'>\lambda>0$:
\begin{enumerate}
\item There is no thickening which commutes with $d$ linearly independent shifts of $\R^d$.
\item For every $(d-1)$-dimensional linear space of shifts of $\R^d$, there exists a thickening which commutes with that space.
\end{enumerate}
\end{theorem}

In the discrete setting, we may also generalize Theorem~\ref{discrete}. This generalization is much easier and we include it here for the sake of completeness.

\begin{theorem}\label{disc_gen}
For any $d\ge 1$ and any pair of intensities $0<p<p'<1$:
\begin{enumerate}
\item If $p'<1-p$, there is no thickening which commutes with any $d$-dimensional lattice of shifts of $\Z^d$.
\item If $1-p<p'$, there is a thickening which commutes with all the shifts of $\Z^d$.
\item For every $(d-1)$-dimensional lattice of shifts of $\Z^d$, there exists a thickening which commutes with that space.
\end{enumerate}
\end{theorem}

As in the one-dimensional setting, the case $p'=1-p$ (for $p<\frac{1}{2}$) appears to have not been treated.

\subsection{General setup and conjecture}

We may consider our positive results as special cases of a much more
general setup. Given two probability measures $\P$ and $\Q$ on two
standard Borel spaces $S$ and $T$ (if the measures are atomless we
might as well take them both to be the uniform measure on the unit
interval $[0,1]$) and a Borel measurable relation $R\subset S \times
T$, we want to know whether there exists a \emph{deterministic
coupling satisfying $R$}, i.e. a function $f:S \to T$ such that
$f(X)\sim \Q$ whenever $X\sim \P$ and the relation is a.s.
satisfied, $\P((X,f(X))\in R)=1$.

\medskip
\noindent{\bf Question:} For which $R$ does such a deterministic coupling exist?
\medskip

Of course, in order for such a deterministic coupling to exist we
need to require that some (not necessarily deterministic) coupling
satisfying $R$ exists. For this, it is clearly necessary that Hall's
condition holds: for any Borel measurable $A\subset S$ we have
$\Q(R(A))\ge \P(A)$, where $R(A)$ is the image of $A$ under $R$,
i.e. $R(A)=\{y\in T \mid \exists x\in A \, (x,y)\in R\}$. Note that
$R(A)$ might fail to be Borel measurable, but it is analytic and
hence universally measurable, so $\Q(R(A))$ is well defined. Under
suitable topological conditions on $S,T$ and $R$, this is also
sufficient, see \cite[Theorem 3.1]{Artstein} (it is not always
sufficient, e.g., if $S=T=[0,1]$, $\P=\Q=U([0,1])$ and $R=\{(x,y)\
|\ x<y\}$).

However, there are still cases in which a coupling exists but no
deterministic coupling exists. As an example, when $\P$ and $\Q$ are
uniform on $[0,1]$ and $R=\big\{(x,y)\in[0,1]^2\ |\ x=2y \bmod
1\big\}$, taking $Y$ to be uniform and $X=2Y \bmod 1$ yields a
coupling satisfying $R$, but it is easily verified that no
deterministic coupling satisfying $R$ exists.

Notice that in this example, while the measures $\P$ and $\Q$ are atomless, in the given coupling $(X,Y)$ the distribution of $X$ conditioned on $Y$ has atoms (in fact, it is atomic). We say that a coupling $(X,Y)$ is \emph{conditionally atomless} if the distribution of $X$ conditioned on $Y$ is atomless for almost all values of $Y$.

\medskip
\noindent{\bf Conjecture:} Given a relation $R$, if there exists a conditionally atomless coupling satisfying $R$, then there also exists a deterministic coupling satisfying $R$.
\medskip

For example, Theorem~\ref{discrete} confirms the conjecture in the special case where $S=T=\{0,1\}^\Z$, $R=\{(a,b)\in S\times T\ |\ \forall n\ a_n\le b_n\}$, $\P=\P_{p}^\Z$ and $\Q=P_{p'}^\Z$ in which case a conditionally atomless coupling is given by $X_n:=Y_n Z_n$ where $Y\sim\P_{p'}^\Z$ and $Z\sim\P_{p/p'}^\Z$ independently.

Another special case of this conjecture has been resolved by Bollob\'as and Varopoulos \cite{BV} who show that this conjecture holds when $\Q$ is purely atomic (in which case having a conditionally atomless coupling is equivalent to $\P$ itself being atomless). Our own methods may be generalized to some cases in which the given spaces $S,T$ are product spaces and the relation $R$ is a product relation. Additional examples can be adapted from the results of Angel, Holroyd and Soo \cite{AHS} on Poisson thinning in finite volume. However, the general case remains open.

We are unsure whether additional topological assumptions are required in the conjecture (as in Hall's condition). For example, one may need to assume that $S$ and $T$ are Polish and the relation $R$ is closed. All our examples except the Bollob\'as-Varopoulos theorem are of this type.

\section{Non-existence of equivariant thickenings}

\begin{proof}[Proof of Theorem~\ref{no}]
For simplicity, we will set $2=\lam'>\lam=1$, but the proof works just as well for any $\lam'>\lam>0$.

Assume, in order to obtain a contradiction, that there exists an equivariant thickening $f:\M\to\M$. Let $X$ be a Poisson process of intensity 1 and let $Y:=X\cup f(X)$. We assume, WLOG, that $X \cap f(X) = \emptyset$, since we can always replace $f(X)$ by $f'(X)=f(X)\setminus X$, which also satisfies $Y=X\cup f'(X)$. Now split $Y$ into two disjoint sets $Y=Y_1 \cup Y_2$ by randomly and independently assigning each point of $Y$ to $Y_1$ or $Y_2$ with probability $\frac12$.

Let $Z$ be a Poisson process of intensity 2, and split it similarly into $Z_1$ and $Z_2$. The resulting distribution on $(Z_1, Z_2)$ is simply the distribution of 2 independent Poisson processes of intensity 1. We will show that the distribution of $(Y_1,Y_2)$ differs from that of $(Z_1,Z_2)$ by constructing an event which has different probabilities under these two distributions. This will imply that the distribution of $Y=X\cup f(X)$ is different from that of $Z$, since the splitting process is the same.

Informally, we consider the possibility that the splitting $(Y_1,Y_2)$ of $Y$ coincides with $(X,f(X))$ on some large interval. On this event, there is another large interval on which $f(Y_1)$ is much more correlated with $Y_2$ than what we would get for $f(Z_1)$ and $Z_2$. The equivariance condition enters in ensuring that the probability of this event decays only exponentially fast in the length of the interval on which the correlation holds.

For a Borel $S\subset \R$, let $\cF_S$ denote the $\sigma$-algebra of the restriction of $\M$ to $S$, i.e. all the events which depend only on the points of the process which are in $S$.

For $\eps>0$, $t\in\R$ and $U\in \M$, let
$$
A_t(U) = A_{t,\eps}(U) = \begin{cases} 1&U\cap [t,t+\eps]\neq\emptyset\\
0& \text{otherwise}\end{cases}.
$$

\begin{claim}
As $\eps \ra 0$ we have
$$\E [A_0(f(X))]=\eps(1+o(1)) \, .$$
\end{claim}
\begin{proof}
On the one hand
\begin{equation*}
\E [A_0(f(X))]\le
\E|f(X)\cap[0,\eps]|=\E|Y\cap[0,\eps]|-\E|X\cap[0,\eps]|=\eps \ .
\end{equation*}
On the other hand
\begin{equation*}
\E [A_0(f(X))]\ge\E [A_0(Y)]-\E [A_0(X)]=\eps(1+o(1)) \ .\qedhere
\end{equation*}
\end{proof}

We continue under the assumption that $\eps>0$ is small enough so
that $\eps/2 < \E [A_0(f(X))] < 2 \eps$.

Now, the random variable $A_0(f(X))$ can be $(\eps/4)$-approximated
by some indicator random variable $B_0(X)$ measurable with respect
to $\cF_{[-r,r]}$ for some $r<\infty$, i.e. $B_0(X)$ takes only the
values 0 and 1, depends only on the points of $X$ in $[-r,r]$ and
\begin{equation*}
  \P(A_0(f(X))\neq B_0(X))<\eps/4\,.
\end{equation*}
In particular $\E[B_0(X)] \le 4\eps$. We note that $B_0$, as a
function on $\M$, is defined only up to null sets of the Poisson
distribution with intensity 1, and define it on all of $\M$ by
choosing some arbitrary representative. Defining
$C_0(U)=A_0(f(U))\cdot B_0(U)$ we have that
\begin{equation*}
\E[C_0(X)] \ge \eps/4\,.
\end{equation*}
Defining $B_t(U)=B_0(\sigma_t(U))$ and $C_t(U)=C_0(\sigma_t(U))$,
where $\sigma_t:\R\ra\R$ is translation by $t$, we have by our
equivariance assumption that $B_t(X)$ is measurable with respect to
$\cF_{[-r+t,r+t]}$ and approximates $A_t(f(X))$ similarly.

For some $L$, consider the events $\{B_{i \eps}(U)\}_{i=0}^{L-1}$
and note that they all belong to $\cF_{[-r,L \eps +r]}$. Let
\begin{equation*}
  b(U)=\sum_{i=0}^{L-1} B_{i\eps}(U)\, ,
\end{equation*}
so that $\E[b(X)]=\E[B_0(X)] L \le 4 \eps L$. By ergodicity of $X$
with respect to the shift by $\eps$, we get that
\begin{equation}\label{eq:ergodic_b}
\P(b(X)<5 \eps L )\ra 1\text{ as }L\ra \infty\,.
\end{equation}

Similarly, letting
\begin{equation*}
  c(U)=\sum_{i=0}^{L-1} C_{i\eps}(U)\, ,
\end{equation*}
then $\P(c(X)>\eps L /8)\ra 1$ as $L\ra\infty$ (although these
events do not necessarily belong to $\cF_{[-r,L \eps +r]}$).

For $\eps>0$ and $U,V \in \M$, let $D_t(U,V)=B_t(U) \cdot A_t(V)$
and let
\begin{equation*}
  d(U,V)=\sum_{i=0}^{L-1} D_{i\eps}(U,V)\, .
\end{equation*}
Notice that $D_t(U,f(U))=B_t(U) \cdot A_t(f(U))=C_t(U)$. In
particular,
\begin{equation}\label{eq:ergodic_d}
  \P(d(X,f(X))> L \eps/8) \ra 1\text{ as }L\ra\infty\, .
\end{equation}

Finally, let
\begin{equation*}
  E(U,V) = \begin{cases} 1&b(U)<5\eps L\text{ and } d(U,V)>\eps
  L/8\\ 0&\text{otherwise}\end{cases}.
\end{equation*}

We claim that this event distinguishes between $(Y_1,Y_2)$ and
$(Z_1,Z_2)$. Informally, this follows from the fact that since
$B_t(Y_1)$ (unlike $A_t(f(Y_1))$ which it approximates) is a
function of $Y_1|_{[t-r,t+r]}$ and hence on the event that
$Y_1|_{[-r,L\eps+r]}=X|_{[-r,L\eps+r]}$, an event whose probability
is only exponentially small in $L\eps+2r$, the probability of
$E(Y_1,Y_2)$ is relatively high. On the other hand, since $Z_1$ and
$Z_2$ are independent, it is very unlikely that $B_t(Z_1)$ and
$A_t(Z_2)$ will occur simultaneously for many times $t$ and
consequently, for suitable choices of $L$ and $\eps$, the
probability of $E(Z_1,Z_2)$ becomes much smaller than that of
$E(Y_1,Y_2)$. More formally, the theorem follows from the next two
claims.

\begin{claim} For every $\eps>0$ we have
\begin{equation*}
\E[E(Y_1,Y_2)]\ge 2^{-1-3(L \eps + 2 r)}
\end{equation*}
if $L$ is sufficiently large.
\end{claim}

\begin{proof}
Define
\begin{equation*}
\begin{split}
\Omega_1&:=\{X\cup f(X)\text{ has at most $3(L \eps + 2 r)$ points in $[-r,L\eps+r]$}\},\\
\tau&:=\min\{t\ge -r\ |\ X\cup f(X)\text{ has $\lfloor 3(L\eps+2r)\rfloor$ points in $[-r,t]$}\},\\
\Omega_2&:=\{Y_1|_{[-r,\tau]}=X|_{[-r,\tau]}\}.
\end{split}
\end{equation*}
We note that $\Omega_2$ depends only on the decisions of the
splitting process for the first $\lfloor 3(L\eps+2r)\rfloor$ points
to the right of $-r$, hence $\Omega_2$ is independent of $X$ and
\begin{equation}\label{eq:Omega_2}
  \P(\Omega_2) = 2^{-\lfloor 3(L\eps+2r)\rfloor} \ge 2^{-3(L \eps + 2
r)} \, .
\end{equation}

In addition, we note that on the events $\Omega_1$ and $\Omega_2$, we have $(Y_1|_{[-r,L\eps+r]},Y_2|_{[-r,L\eps+r]})=(X_{[-r,L\eps+r]}, f(X)|_{[-r,L\eps+r]})$. We conclude that
\begin{equation}\label{eq:E_cond}
\begin{split}
&\E[E(Y_1,Y_2)\ |\ \Omega_1, \Omega_2]=\E[E(X,f(X))\ |\
\Omega_1,\Omega_2] = \E[E(X,f(X))\ |\ \Omega_1] \to 1
\end{split}
\end{equation}
as $L\to\infty$ since $\P(\Omega_1) \ra 1$ by ergodicity and
$\E[E(X,f(X))]\ra 1$ by \eqref{eq:ergodic_b} and
\eqref{eq:ergodic_d}. Since $\Omega_1$ and $\Omega_2$ are
independent we have
\begin{equation*}
  \E[E(Y_1,Y_2)] \ge \P(\Omega_1)\P(\Omega_2)\E[E(Y_1,Y_2)\ |\ \Omega_1,
  \Omega_2].
\end{equation*}
Thus the claim follows from \eqref{eq:Omega_2} and
\eqref{eq:E_cond}.
\end{proof}

\begin{claim} For every $\eps>0$ and $L\ge 1$ we have
\begin{equation*}
\E[E(Z_1,Z_2)]\le 2^{5\eps L}\eps^{\eps L/8}.
\end{equation*}
\end{claim}
\begin{proof}
Let $T=\{0\le i\le L-1\ |\ B_{i\eps}(Z_1)=1\}$ so that $|T| =
b(Z_1)$. Since $Z_1, Z_2$ are independent Poisson processes of
intensity 1 and by the definition of $A_t$, we have
\begin{equation*}
\begin{split}
\P(&d(Z_1, Z_2)\ge m\ |\ T)=\P\big(\exists R\subseteq T, |R|=m\text{ such that }\prod_{i\in R} A_{i\eps}(Z_2)=1\ |\ T\big)\le \\
&\le \sum_{R\subseteq T, |R|=m} \P\big(\prod_{i\in R}
A_{i\eps}(Z_2)=1\ |\ T\big) = \sum_{R\subseteq T, |R|=m}
\P\big(\prod_{i\in R} A_{i\eps}(Z_2)=1\big) \le 2^{|T|}\eps^m.
\end{split}
\end{equation*}

Therefore,
\begin{equation*}
  \P(d(Z_1,Z_2)>\eps L/8\ |\ b(Z_1) < 5\eps L) \le 2^{5 \eps L} \eps^{\eps L
  /8}.
\end{equation*}
Thus the claim follows from the definition of $E(Z_1,Z_2)$.
\end{proof}

Comparing the estimates of the last two claims for small enough
$\eps>0$ and large enough $L$ shows that $(Y_1,Y_2)$ and $(Z_1,Z_2)$
do not have the same distribution, yielding a contradiction to the
existence of $f$.
\end{proof}

\section{Existence of non-equivariant thickenings}\label{non_existence_proof_sec}

The proofs of Theorems \ref{yes} and \ref{discrete} are essentially the same. We will first prove Theorem \ref{discrete} and then discuss the changes needed to prove Theorem \ref{yes}.

Let $0<p<p'<1$ be fixed. For $0\le r\le 1$, denote by $\P_r$ the distribution of a $\{0,1\}$-valued random variable with expectation $r$, and let $\P_r^I$ be a set of i.i.d. $\P_r$ random variables, indexed by $I$. Our goal is to construct a measurable $f:\{0,1\}^\Z \ra \{0,1\}^\Z$ such that if $X \sim \P_p^\Z$ then $f(X) \sim \P_{p'}^\Z$ and a.s. for all $i\in\Z$ we have $f(X)_i \ge X_i$.

Since we don't require equivariance, the specific choice of index set plays no role beyond its cardinality. That is, there is no difference between a (discrete) thickening on $\{0,1\}^\Z$, as in the statement of Theorem~\ref{discrete}, and a thickening on $\{0,1\}^\N$ or $\{0,1\}^{\N\times \N}$ (which are defined analogously). To be more specific, let $n:\N\times \N \ra \N$ be a bijection and let $h:\{0,1\}^{\N} \ra \{0,1\}^{\N \times \N}$ be the isomorphism defined by $h(X)_{ij}=X_{n(i,j)}$. If $f$ is a thickening of $\P_p^\N$ into $\P_{p'}^\N$, then $h \circ f \circ h^{-1}$ is a thickening of $\P_{p}^{\N\times\N}$ into $\P_{p'}^{\N\times\N}$ and vice versa.

Another useful fact is that $\P_r^\N$ and $\P_s^\N$ are isomorphic (as measure spaces), for any $0<r,s<1$. Let $g$ be such an isomorphism taking $\P_{\frac12}^\N$ into $\P_{q}^\N$, where $q:=\frac{p'-p}{1-p}$. $q$ is chosen so that if $x\sim \P_p$ and $y\sim \P_q$ are independent then $\max(x,y)\sim \P_{p'}$.

We define an \emph{extractor} to be a function $f:\{0,1\}^\N \ra \{0,1\}$ such that if $X\sim \P_{p}^\N$ and $Y \sim \P_{q}^\N$ are independent then
\begin{align*}
f(X) &\sim \P_\frac12\qquad\text{ and}\\
f(X)\mbox{ and }\max&(X,Y)\mbox{ are independent.}
\end{align*}
where $\max(X,Y)$ is taken coordinate-wise. We remark that this is different from the extractor which is sometimes used in the computer science literature.

How are extractors useful? First, notice that given independent $X\sim\P_p^\N$ and $Y\sim\P_q^\N$, by rearranging indices (using the function $n$ above) one can extract infinitely many bits from $X$, i.e. one can get a function $f:\{0,1\}^\N \ra \{0,1\}^\N$, such that $f(X) \sim \P_{\frac12}^\N$ independently of $\max(X,Y)$. Second, by applying $g$ we can get a sequence distributed $\P_{q}^\N$. Now, to thicken $\P_{p}^{\N\times\N}$ into $\P_{p'}^{\N\times\N}$, define $F:\{0,1\}^{\N\times\N}\to\{0,1\}^{\N\times\N}$ by
$$F(X)^i=\max(X^i,g(f(X^{i+1})))$$
where for $U\in\{0,1\}^{\N\times\N}$ we write $U^i$ for $U(i,\cdot)$.

\begin{claim} \label{thickening}
If $f$ is an extractor, $F$ is a thickening.
\end{claim}
\begin{proof}
First note that for each $i\in\N$, $F(X)^i\sim\P_{p'}^\N$ by definition of $f$ and $g$ and since $X^i$ and $X^{i+1}$ are independent. Thus, to prove the claim, it is sufficient to show that for every integer $j\ge 1$,
\begin{equation}\label{joint_F_independence}
\text{$\{F(X)^1, F(X)^2, ..., F(X)^{j-1}, F(X)^j\}$ are jointly independent.}
\end{equation}

We first claim that for each integer $j\ge 1$,
\begin{equation}\label{f_F_joint_independence}
\text{$\{X^1, X^2, ..., X^{j-1}, f(X^j), F(X)^j\}$ are jointly independent.}
\end{equation}
To see \eqref{f_F_joint_independence}, note that since $\{X^i\}_{i\ge1}$ are jointly independent and $(f(X^j), F(X)^j)$ is measurable with respect to $(X^i)_{i\ge j}$, it is sufficient to show that $f(X^j)$ is independent from $F(X)^j$. This follows from the definition of extractor.

We now prove \eqref{joint_F_independence} by induction on $j$. For $j=1$ there is nothing to prove. Assume \eqref{joint_F_independence} holds for $j=k-1$ and let us prove it for $j=k$. Since $(F(X)^i)_{1\le i\le k-1}$ is measurable with respect to $(X^1, X^2, ..., X^{k-1}, f(X^k))$, it follows from \eqref{f_F_joint_independence} that $(F(X)^i)_{1\le i\le k-1}$ is independent from $F(X)^k$. It remains to show that $\{F(X)^i\}_{1\le i\le k-1}$ are jointly independent which follows from our induction hypothesis.\qedhere

%
\end{proof}

All that is left, then, is to construct an extractor. Unfortunately, such an object does not exist.

\begin{lemma} \label{no-extractor}
There is no extractor.
\end{lemma}
\begin{proof}
Let $X\sim \P_{p}^\N$ and $Y \sim \P_{q}^\N$ be independent and define $Z:=\max(X,Y)$. Assume that $f$ is an extractor. We will reach a contradiction by showing that $f(X)$ is independent of $\{X_i\}_{i=1}^k$ for any integer $k$.

Fix $k\ge 1$. For $U\in\{0,1\}^{\N}$, let $A(U)$ be the event $\wedge_{i\le k} (U_i=0)$. Since $f$ is an extractor, $f(X)|A(Z) \sim \P_{\frac12}$, but $A(Z)=A(X)\wedge A(Y)$, and $X$ and $Y$ are independent, so the distribution of $X|A(Z)$ is the same as $X|A(X)$, so $f(X)|A(X) \sim \P_{\frac12}$.

Now, for $1\le j\le k$, let $A_j(U)$ be the event $\wedge_{i\le k , i\ne j} (U_i=0) \wedge (U_j=1)$. Again, $f(X)|A_j(Z) \sim \P_{\frac12}$, but now $X|A_j(Z)$ is $\frac{p}{p'} X|A_j(X) + (1-\frac{p}{p'}) X| A(X)$ (that is, $X_i=0$ for $i\le k, i\neq j$ and $X_j\sim\P_{p/p'}$), since $\P(X_j=1|Z_j=1)=\frac{p}{p'}$. We already know that $f(X)|A(X)\sim \P_{\frac12}$, so we conclude that also $f(X)|A_j(X)\sim \P_{\frac12}$.

Proceed by induction on the number of $1$'s among $\{Z_i\}_{i=1}^k$ to show that conditioned on any sequence of values for $\{X_i\}_{i=1}^k$, $f(X)$ is distributed $\P_{\frac12}$.
\end{proof}

Fortunately, one can make do with something that is only almost an extractor, though the way it is used will be a bit more complicated. An $\eps$-\emph{extractor} is a function $f:\{0,1\}^\N \ra \{0,1\}$ such that if $X\sim \P_p^\N$ and $Y \sim \P_q^\N$ are independent then
\begin{align*}
f(X)&\sim \P_{\frac12}\qquad\text{ and}\\
\E d_{\text{TV}}\Big(\LL\big(f(X)\ \big|\ &\max(X,Y)\big),\P_\frac12\Big)<\eps,
\end{align*}
 where $\LL(f(X)\ |\ \max(X,Y))$ is the law of $f(X)$ conditioned on $\max(X,Y)$ and $d_\text{TV}(\LL_1,\LL_2)$ is the total variation distance between the laws $\LL_1$ and $\LL_2$. That is, observing $\max(X,Y)$ gives us little information on $f(X)$. Learning from our previous experience, we first verify the existence of $\eps$-extractors.

\begin{lemma} \label{yes-extractor}
For any $\eps>0$ there is an $\eps$-extractor.
\end{lemma}
\begin{proof}
Fix $\eps>0$. For an integer $k\ge1$, let $a_k$ be the parity of the first $k$ values of $X$, i.e. $a_k:=\bigoplus_{i=1}^k X_i$. Let $\ell_k:=\sum_{i=1}^k \max(X_i,Y_i)$. Then it is readily verified, using the Fourier transform, that $\E(a_k|\max(X,Y)) \sim \P_{(1-(1-2\frac{p}{p'})^{\ell_k})/2}$. Let $\ell'$ be such that $\left|1-2\frac{p}{p'}\right|^{\ell'} < \eps/2$ and fix $k$ large enough so that $\P(\ell_k<\ell') < \eps/2$. Then, observing that $d_{\text{TV}}\Big(\LL\big(a_k\ |\ \max(X,Y)\big),\P_\frac12\Big)=|\P\big(a_k=1\ |\ \max(X,Y)\big)-\frac{1}{2}|$, we have
\begin{equation*}
\E d_{\text{TV}}\Big(\LL\big(a_k\ |\ \max(X,Y)\big),\P_\frac12\Big)< \P(\ell_k \ge \ell') \eps/2 + \P(\ell_k<\ell') 1 \le \eps.
\end{equation*}

Hence, $a_k$ satisfies the second requirement of $\eps$-extractor. To get the first requirement, let $m=m(X)$ be the minimal positive integer such that $X_{k+2m}\ne X_{k+2m+1}$ and let $b_k:=X_{k+2m}$. Then $b_k\sim\P_{\frac12}$ and is independent of $a_k$, both unconditionally and conditionally on $\max(X,Y)$. Therefore, $f(X):=a_k\oplus b_k$ (where $a_k\oplus b_k$ is defined to be 1 iff $a_k$ is different from $b_k$) satisfies both requirements of being an $\eps$-extractor.
\end{proof}

Of course, we cannot simply replace the extractors from the proof of Claim~\ref{thickening} with $\eps$-extractors, since one might learn something about the output bits of the $\eps$-extractors by observing the thickening of the bits from which they were extracted. We will therefore introduce a ``correction'' mechanism for these bits.

Given an $\eps$-extractor $f$, a \emph{corrector} for $f$ is a function $f':\{0,1\}^\N \times \{0,1\}^\N \times\{0,1\}^\N \ra \{0,1\}$ such that when $X\sim \P_p^\N$ , $Y\sim \P_q^\N$ and $Z \sim \P_\frac12^\N$ are independent, the following properties hold:
\begin{align}
f(X) &\oplus f'(X,Y,Z)  \sim \P_\frac12,\label{corrector_prop}\\
f(X) \oplus f'(X,Y,Z) &\mbox{ and } \max(X,Y) \mbox{ are independent},\label{corrector_ind_prop}\\
\E(f'&(X,Y,Z)) < \eps\label{corrector_dist_prop}
\end{align}
(where, again, $a\oplus b$ is defined to be 1 iff $a$ is different from $b$)
\begin{claim} \label{corrector}
For any $\eps$-extractor, there is a corrector.
\end{claim}
\begin{proof}
Let $f$ be an $\eps$-extractor and $X\sim \P_p^\N$ , $Y\sim \P_q^\N$
and $Z \sim \P_\frac12^\N$ be independent. Define
$g:\{0,1\}\times\{0,1\}^\N\to [0,1]$ by
\begin{equation*}
  g(s,m) = \P\left(f(X)=s\ \big|\ \max(X,Y)=m\right).
\end{equation*}
Let $U:\{0,1\}^\N \ra [0,1]$ be defined by $U(z):=\sum_{i=1}^\infty
z_i 2^{-i}$ so that $U(Z)$ is a uniform random variable on $[0,1]$.
Now for $x,y,z\in \{0,1\}^\N$, we define $f'$ as
$$ f'(x,y,z)=\left\{\begin{array}{lll}
1 & U(z) g(f(x),\max(x,y)) > \frac12 \\
0 & \mbox{otherwise}
\end{array}\right.$$

Let us motivate informally the definition of $f'$. For $m\in\{0,1\}^\N$, consider the event $\Omega_m=\{\max(X,Y)=m\}$. Given $\Omega_m$ we have the random variable $f(X) | \Omega_m$ whose distribution is close to $\P_{\frac12}$ and we want that $f(X)\oplus f'(X,Y,Z) | \Omega_m$ will be exactly $\P_{\frac12}$. Furthermore, we want $\E(f'(X,Y,Z)\ |\ \Omega_m)$ to be small. Now, if $\alpha\ge \frac{1}{2}$ and $a \sim \P_\alpha$ we have $d_{TV}(\LL(a),\P_{\frac12})=\alpha-\frac12$. Defining $b$ to be 0 if $a=0$ and 1 with probability $1-\frac{1}{2\alpha}$ if $a=1$, it is easy to check that $a \oplus b \sim \P_{\frac12}$ and that $\E(b)=\alpha-\frac12$, which is the minimal possible given that $a \oplus b \sim \P_{\frac12}$. This is exactly what the above definition does (in an analogous way for the case $\alpha<\frac12$), where the extra independent randomness is provided by $Z$.



Indeed, to verify formally that $f'$ is a corrector for $f$, we fix
$m \in \{0,1\}^\N$ and define $\alpha:=g(0,m)$. If $\alpha\ge
\frac{1}{2}$ then
\begin{equation*}
\begin{split}
\P&\left(f(X)=0 ,\ f'(X,Y,Z)=0 \ \big| \ \max(X,Y)=m\right)=\\
&=\alpha\cdot\P\left(f'(X,Y,Z)=0\ \big|\ f(X)=0,\
\max(X,Y)=m\right)=\\
&=\alpha\cdot\P\left(U(Z)\, g(f(X),\max(X,Y))\le \frac{1}{2}\ \big|\
f(X)=0,\
\max(X,Y)=m\right)=\\
&=\alpha\cdot\P\left(U(Z)\, g(0,m)\le \frac{1}{2}\right)=\alpha\cdot\P\left(U(Z)\alpha\le \frac{1}{2}\right)=\alpha\cdot \frac{1}{2 \alpha} = \frac12,\\
\end{split}
\end{equation*}
and similarly
\begin{equation*}
\begin{split}
&\P\left(f(X)=0 ,\ f'(X,Y,Z)=1 \ \big| \ \max(X,Y)=m\right)= \alpha(1-\frac{1}{2\alpha})=\alpha-\frac{1}{2},\\
&\P\left(f(X)=1 ,\ f'(X,Y,Z)=0 \ \big| \ \max(X,Y)=m\right)= (1-\alpha)\cdot 1=1-\alpha.
\end{split}
\end{equation*}
Thus
\begin{equation*}
\begin{split}
&\P\left(f(X)\oplus f'(X,Y,Z) = 0\ \big|\ \max(X,Y)=m\right) = \frac{1}{2}\quad\text{and}\\
&\P\left(f'(X,Y,Z) = 1\ \big|\ \max(X,Y)=m\right)=\left|\alpha-\frac{1}{2}\right|.
\end{split}
\end{equation*}
These two equalities follow analogously in the case $\alpha<\frac{1}{2}$. Hence, $f(X)\oplus f'(X,Y,Z) \sim \P_\frac12$, independently of $\max(X,Y)$, verifying \eqref{corrector_prop} and \eqref{corrector_ind_prop}. In addition, since by definition
\begin{equation*}
d_{TV} \left(\LL\left(f(X)\ \big|\ \max(X,Y)=m\right) , \P_\frac12 \right)=\left|\alpha-\frac{1}{2}\right|,
\end{equation*}
we see that
\begin{equation*}
\P(f'(X,Y,Z) = 1)=\E d_{TV} \left(\LL\left(f(X)\ \big|\ \max(X,Y)=m\right) , \P_\frac12 \right)<\eps,
\end{equation*}
since $f$ is an $\eps$-extractor. This verifies \eqref{corrector_dist_prop} and proves the claim.
\end{proof}

As before, we need more than a single bit. An $\eps$-\emph{extractor into} $\{0,1\}^\N$ is a function $f:\{0,1\}^\N \ra \{0,1\}^\N$ such that if $X\sim \P_p^\N$ and $Y \sim \P_q^\N$ are independent then
\begin{align*}
f(X)&\sim \P_{\frac12}^\N\qquad\text{ and}\\
\E d_\text{TV}\Big(\LL\big(f(X)\ &|\ \max(X,Y)\big),\P_\frac12^\N\Big)<\eps.
\end{align*}
To construct an $\eps$-extractor into $\{0,1\}^\N$ we take a sequence of functions $f_i:\{0,1\}^\N \ra \{0,1\}$ such that $f_i$ is an $\eps 2^{-i}$-extractor and define
$$f(X)_i=f_i(h(X)^i) \ ,$$
where we recall that $h$ is an isomorphism taking $\{0,1\}^\N$ into $\{0,1\}^{\N\times\N}$. That this results in an $\eps$-extractor follows easily from the (equivalent) definition of the total variation distance $d_{TV}(\LL_1,\LL_2)$ as the minimum of $\P(X\ne Y)$ over all possible joint distributions $(X,Y)$ where $\LL(X)=\LL_1$ and $\LL(Y)=\LL_2$. Thus, given two infinite sequences of distributions $\LL_1^n$ and $\LL_2^n$, the total variation distance between $\prod_n \LL_1^n$ and $\prod_n \LL_2^n$ is bounded by the sum of distances $\sum_n d_{TV}(\LL_1^n,\LL_2^n)$ since one may take, for each $n$ independently, a coupling $(X_n,Y_n)$ between $\LL_1^n$ and $\LL_2^n$ which minimizes $\P(X_n\ne Y_n)$ and then define $(X,Y)$ as $((X_n,Y_n)_n)$ so that $\P(X \ne Y)\le \sum_n \P(X_n \ne Y_n)$ by a simple union bound.

For $f$, an $\eps$-extractor into $\{0,1\}^\N$, one calls $f':\{0,1\}^\N \times \{0,1\}^\N \times \{0,1\}^\N \ra \{0,1\}^\N$ a \emph{corrector}, if when $X\sim \P_p^\N$, $Y\sim \P_q^\N$ and $Z \sim \P_\frac12^\N$ are independent, the following properties hold:
\begin{align*}
f(X) &\oplus f'(X,Y,Z)  \sim \P_\frac12^\N,\\
f(X) \oplus f'(X,Y,Z) &\mbox{ and } \max(X,Y) \mbox{ are independent},\\
\P(f'(X,Y,Z)&\ne (0,0,\ldots) ) < \eps
\end{align*}
(where $(U\oplus V)_i:=U_i\oplus V_i$).

Existence of correctors can be proved by the methods of Claim~\ref{corrector}. Furthermore, if the $\eps$-extractor is constructed as above, as a sequence of $\eps 2^{-i}$-extractors, then one can take a corresponding sequence of correctors to get a corrector for this $\eps$-extractor.

Given an $\eps$-extractor, $f$, and an associated corrector, $f'$, one defines the \emph{corrected extractor}, $f'':\{0,1\}^\N \times \{0,1\}^\N \times \{0,1\}^\N \ra \{0,1\}^\N$, to be $f''(X,Y,Z):=f(X)\oplus f'(X,Y,Z)$.
Corrected extractors are very similar to extractors. The difference is that they depend, though rather weakly, on extra bits (and also, unlike extractors, they exist).
We need an analogue of Claim~\ref{thickening} for corrected extractors.
\begin{claim}\label{corrected_thickening}
Fix $n\in\N$ and let $X\sim \P_p^{\{1,2,\ldots, n\}\times\N}$, $Y^{n}\sim \P_q^\N$ and $Z^{n} \sim \P_\frac12^\N$ be all jointly independent. For each $1\le i\le n$, let $\eps_i>0$ and let $f''$ be a corrected extractor for some $\eps_i$-extractor.
Define $F(X)^n:=\max(X^n,Y^n)$ and for $1\le i\le n-1$, define (by downward induction on $i$)
\begin{equation*}
\begin{split}
Y^i &:=g(h(f''_{i+1}(X^{i+1},Y^{i+1}, Z^{i+1}))^1), \\
Z^i &:=h(f''_{i+1}(X^{i+1},Y^{i+1}, Z^{i+1}))^2, \\
F(X)^i &:=\max(X^i,Y^i)
\end{split}
\end{equation*}
(where $g$ and $h$ were defined at the beginning of this section).
Then
\begin{equation*}
(F(X)^i)_{i=1}^{n}\sim\P_{p'}^{\{1,2,\ldots, n\}\times\N}.
\end{equation*}
\end{claim}
\begin{proof}
The proof is very similar to the proof of Claim~\ref{thickening}. First note by downward induction on $1\le i\le n$, the joint independence of the $(X^i)$ and the properties of corrected extractors that $(X^1,\ldots, X^i, Y^i, Z^i)$ are jointly independent, $Y^i\sim \P_q^\N$ and $Z^i \sim \P_\frac12^\N$. Thus $F(X)^i\sim\P_{p'}^\N$ for each $1\le i\le n$ and, by the properties of corrected extractors, $(X^1,\ldots, X^i, Y^i, Z^i, F(X)^{i+1})$ are jointly independent for $1\le i\le n-1$. Since $(F(X)^j)_{j=1}^i$ are measurable with respect to $(X^1,\ldots, X^i, Y^i, Z^i)$, we deduce that $(F(X)^i)_{i=1}^n$ are jointly independent, as required.
\end{proof}

We are now prepared to prove our theorem.

\begin{proof}[Proof of Theorem~\ref{discrete}]
First, by using $h$ we transfer the problem to thickening $\P_{p}^{\N\times\N}$ into $\P_{p'}^{\N\times\N}$.

For $i\in \N$, let $f_i$ be a $\frac{1}{2^i}$-extractor into $\P_\frac12^{\N}$, Let $f'_i$ be a corresponding corrector, and let $f''_i$ be the resulting corrected extractor.

We would like to make the following definitions: for $i\in\N$
\begin{equation*}
\begin{split}
Y^i &:=g(h(f''_{i+1}(X^{i+1},Y^{i+1}, Z^{i+1}))^1), \\
Z^i &:=h(f''_{i+1}(X^{i+1},Y^{i+1}, Z^{i+1}))^2, \\
F(X)^i &:=\max(X^i,Y^i). \\
\end{split}
\end{equation*}
Then Claim~\ref{corrected_thickening} would show $F$ is a thickening. Alas, this is not well defined since for each $i$, $(Y^i,Z^i)$ depend on $(Y^{i+1},Z^{i+1})$ and so on ad infinitum. However, since corrections are rare, there is a way to make sense of the above definitions, as follows.

For $n \in \N$ define $Y_n$ and $Z_n$ by
\begin{align}\label{Y_n_def}
(Y_n)^i&:=\left\{\begin{array}{ll}
g(h(f_{i+1}(X^{i+1}))^1) & \mbox{if $i\ge n$}\\
g(h(f''_{i+1}(X^{i+1},(Y_n)^{i+1},(Z_n)^{i+1}))^1) & \mbox{if $i<n$}
\end{array}\right.\\\label{Z_n_def}
(Z_n)^i&:=\left\{\begin{array}{ll}
h(f_{i+1}(X^{i+1}))^2 & \mbox{if $i\ge n$}\\
h(f''_{i+1}(X^{i+1},(Y_n)^{i+1},(Z_n)^{i+1}))^2 & \mbox{if $i<n$}
\end{array}\right.
\end{align}
In other words, we use $f$ (without correction) instead of $f''$ when $i\ge n$. Since $f$ depends only on $X$, this yields, for any $n\in\N$, well defined sequences, $Y_n$ and $Z_n$.

\begin{claim}\label{Y_n_Z_n_convergence_claim}
$Y_n$ and $Z_n$ a.s. converge (pointwise) as $n\ra \infty$ to limits $Y$ and $Z$ satisfying for each $i\in\N$:
\begin{enumerate}
\item $(X^1,\ldots, X^i, Y^i, Z^i)$ are jointly independent.
\item $Y^i\sim \P_q^\N$ and $Z^i \sim \P_\frac12^\N$.
\item $Y^i=g(h(f''_{i+1}(X^{i+1},Y^{i+1}, Z^{i+1}))^1)$ and $Z^i=h(f''_{i+1}(X^{i+1},Y^{i+1}, Z^{i+1}))^2$.
\end{enumerate}
\end{claim}

The theorem follows from this claim, since letting $F(X)^i :=\max(X^i,Y^i)$ we obtain that $F$ is a thickening by Claim~\ref{corrected_thickening}.\qedhere

\begin{proof}[Proof of Claim~\ref{Y_n_Z_n_convergence_claim}]
The first two properties of the Claim hold for $(Y_n,Z_n)$ by their definition and the properties of corrected extractor and hence will hold for any possible limit of $(Y_n,Z_n)$. To see that $(Y_n,Z_n)$ converge and to check the third property in the Claim, we consider the probability that $(Y_n, Z_n) = (Y_{n+1}, Z_{n+1})$.


First, notice that $\big((Y_n)^i, (Z_n)^i\big)=\big((Y_{n+1})^i, (Z_{n+1})^i\big)$ for any $i>n$ and if $\big((Y_n)^n, (Z_n)^n\big)=\big((Y_{n+1})^n, (Z_{n+1})^n\big)$ then we have $\big((Y_n)^i, (Z_n)^i\big)=\big((Y_{n+1})^i, (Z_{n+1})^i\big)$ for all $i$, by backward induction on $i$.

Using that $f_{n+1}$ is a $2^{-(n+1)}$-extractor, $f_{n+1}'$ is a corrector for $f_{n+1}$ and the definition of the corrected extractor $f_{n+1}''$, we get for each $n\in\N$,
\begin{equation*}
\begin{split}
\P\Big(\big((Y_n&)^n, (Z_n)^n\big)\neq\big((Y_{n+1})^n,(Z_{n+1})^n\big)\Big) = \\
&=\P(f_{n+1}(X^{n+1})\neq f''_{n+1}(X^{n+1},(Y_{n+1})^{n+1},(Z_{n+1})^{n+1}))=\\
&=\P(f'_{n+1}(X^{n+1},(Y_{n+1})^{n+1},(Z_{n+1})^{n+1})\ne (0,0,\ldots)) < \frac{1}{2^{n+1}}.
\end{split}
\end{equation*}
The sum of these probabilities is finite and hence, there exists a.s. an $m\in \N$ such that $\big(Y_n,Z_n\big)=\big(Y_{n+1},Z_{n+1}\big)$ for all $n>m$. Thus $(Y_n, Z_n)$ converge a.s. and the third property of the Claim holds for the limit since, by definition, it holds for $(Y_n^i, Z_n^i)$ for $n>i$.
\end{proof}
\end{proof}

To adapt this argument to prove Theorem~\ref{yes} one needs to construct an $\eps$-extractor from a Poisson process (instead of from $\{0,1\}^\N$). To do this let $a$ be the parity of the number of points in $X|_{[-r,r]}$ and let $b:=\frac{1}{2}+\frac{1}{2}\sgn(\min(X|_{(r,\infty)})+\max(X|_{(-\infty,-r)}))$. Then for $r$ large enough $a\oplus b$ is an $\eps$-extractor. Note that Lemma~\ref{no-extractor} also holds in this context; the proof is by induction on the number of points of $Z|_{[-r,r]}$.

Two other ingredients are needed: The first, a (measure space) isomorphism \quad $h:\M \ra \M^\N$ taking a Poisson process of intensity 1 into countably many independent Poisson processes of intensity 1, can be induced from an isomorphism $n:\R \ra \R\times\N$. The second is an isomorphism $g:\{0,1\}^\N \ra \M$, taking $\P_{\frac12}^\N$ into a Poisson process of intensity 1.

{\bf Remark:}
Note that the proof shows that we may obtain, in addition to the thickened process $F(X)$, infinitely many extra $\P_{\frac12}$ bits which are functions of $X$ and independent of $F(X)$ (for example, we may take $h(f''_{i+1}(X^{i+1},Y^{i+1}, Z^{i+1}))^3$ where $f''$, $Y$ and $Z$ are as defined in the proof of Theorem~\ref{discrete}). This will be useful in the proof of Theorem~\ref{cont_gen}.


\section{Higher dimensions} \label{highdim}

\begin{proof}[\textbf{Proof of Theorem~\ref{disc_gen}}]
The first part follows from entropy considerations, just as in the one-dimensional case.

For the second part, one may simply partition $\Z^d$ into ``fibers'' of the form $(x+ke_1)_{k\in\Z}$, where $e_1=(1,0,\ldots,0)$, and apply the one-dimensional discrete equivariant thickening constructed in \cite{B2} to each fiber separately.

The third part follows similarly. If $L$ is a $(d-1)$-dimensional lattice in $\Z^d$, we first choose some $v\in\Z^d$ which is linearly independent of $L$ (over $\Q$). Then we partition $\Z^d$ into ``fibers'' of the form $(x+k v)_{k\in\Z}$. For each fiber $\phi$, we choose, in some arbitrary way, a unique representative $x_0(\phi)\in\phi$ such that if $\phi_1$ and $\phi_2$ are two fibers satisfying $\phi_2=\phi_1+u$ for some $u\in L$, then $x_0(\phi_2)=x_0(\phi_1)+u$ (here, we use the linear independence condition). Finally, we apply the one-dimensional thickening given by Theorem~\ref{discrete} separately on each fiber $\phi$, taking the origin of that fiber to be $x_0(\phi)$.
\end{proof}

\begin{proof}[\textbf{Proof of Theorem~\ref{cont_gen}, part 2}]
First, consider the existence of a thickening equivariant with respect to $d-1$ independent shifts. The proof in this case is similar to the third part of Theorem~\ref{disc_gen}. By applying a linear transformation, we may assume without loss of generality that these shifts are by the first $d-1$ unit vectors. One can then partition $\R^d$, up to measure 0, into strips of the form $[i_1,i_1+1]\times [i_2, i_2 +1] \times \cdots \times [i_{d-1},i_{d-1}+1] \times \R$ (with $i_1,\ldots, i_{d-1}\in\Z$), and use the same non-equivariant thickening in each of these strips. The existence of a non-equivariant thickening in a strip is guaranteed either by constructing it directly, by the methods of Theorem~\ref{yes}, or by noting (as in the introduction) that the strip and $\R$ are isomorphic as measure spaces, and this induces an isomorphism between the corresponding Poisson processes.

It is only slightly harder to see how to construct a thickening equivariant with respect to all shifts in some $(d-1)$-dimensional linear space. Again, we may assume WLOG that the space of shifts is simply $\R^{d-1}$ (the subspace spanned by the first $d-1$ coordinates). Let $X'$ be all the points of $X$ which fall inside the slab $\R\times\R\times\cdots\times\R\times[0,1]$. Let $X''$ be the projection of $X'$ onto $\R^{d-1}$. Use $X''$ to equivariantly partition $\R^{d-1}$, up to measure 0, into countably many cells, e.g. by taking the Voronoi tessellation. Then for each cell $\gamma$ we have that $X$ restricted to each ``strip'' $\gamma \times (\R\setminus [0,1])$ is a Poisson process. We apply a (non-equivariant) thickening to each of these strips, but also extract some extra bits (as in the remark at the end of section~\ref{non_existence_proof_sec}) and use them to add points in $\gamma \times [0,1]$. The resulting function is a thickening and is equivariant with respect to all shifts in $\R^{d-1}$.
\end{proof}

\begin{proof}[\textbf{Proof of Theorem~\ref{cont_gen}, part 1}]
First, let us consider the one-dimensional case, where we weaken the equivariance requirement to integer shifts only. The only place in the proof of Theorem~\ref{no} where the (full) shift equivariance was used was when we showed that for some small enough $\eps>0$, there exists an $r<\infty$ such that each of the events of the form $A_{i \eps}(f(X))$ can be $\eps/4$-approximated by an event $B_{i\eps}(X)$ which belongs to $\cF_{[i\eps-r,i\eps+r]}$. We were able to do that since $B_0(X)$ belonged to $\cF_{[-r,r]}$ and, using shift equivariance, we could choose $B_{i \eps}(X)$ to be a shift of $B_0(X)$.

To get the same using only equivariance w.r.t. integer shifts, we first choose $\eps=1/m$ for some large integer $m$ (this can always be done since all that we required of $\eps$ is to be small). Then for each $A_{i\eps}(f(X))$ for $0\le i<m$ we may find a $B_{i\eps}(X)$ which $\eps/4$-approximates it and belongs to $\cF_{[i\eps-r_i,i\eps+r_i]}$ for some $r_i<\infty$. Then we define $r:=\max\{r_i\}_{i=0}^{m-1}$ and for each $i\ge m$, we $\eps/4$-approximate $A_{i\eps}(f(X))$ by the shift of $B_{(i \mod m) \eps}(X)$ by the integer $(i-(i\mod m))\eps$. Thus, the equivariance w.r.t. integer shifts ensures that $B_{i \eps}(X)$ belongs to $\cF_{[i\eps-r,i\eps+r]}$ for all $i$. The rest of the proof follows as in the proof of Theorem~\ref{no}.

We turn now to the multi-dimensional setting. We first observe that the proof of Theorem~\ref{no} may be adapted in a straightforward manner to the multi-dimensional setting when we have full shift equivariance. To do so, one defines, for $\eps>0$, the events $A_{i_1 \eps,\ldots, i_d \eps}(U)$ (with $i_1,\ldots, i_d\in\Z$) to be ``there is a point of $U$ in $[i_1 \eps,(i_1+1)\eps]\times \cdots \times [i_d\eps,(i_d+1)\eps]$''. Then one needs to show that for some small enough $\eps>0$, there exists an $r<\infty$ such that each of the events $A_{i_1 \eps,\ldots, i_d \eps}(f(X))$ can be $\eps^d/4$-approximated by an event $B_{i_1 \eps,\ldots, i_d \eps}(X)$ which belongs to $\cF_{[i_1\eps-r,i_1\eps+r]\times\cdots\times[i_d\eps-r,i_d\eps+r]}$. This is where shift equivariance is used. In the rest of the proof one proceeds exactly as in the one-dimensional setting (and, in particular, shift equivariance is no longer used) where the events $C_{i\eps}(U)$ and $D_{i\eps}(U,V)$ are replaced by $C_{i_1 \eps,\ldots, i_d \eps}(U)$ and $D_{i_1 \eps,\ldots, i_d \eps}(U,V)$ with analogous definitions and where $b(U)$ is now defined as the number of $B_{i_1 \eps,\ldots, i_d \eps}(U)$ which occur for $0\le i_1,\ldots, i_d\le L-1$ for some large $L$, and $c(U)$ and $d(U,V)$ are defined likewise. $E(U,V)$, $\Omega_1$, $\Omega_2$ are defined analogously.

To adapt this proof to the case of equivariance with respect to $d$ linearly independent shifts of $\R^d$, one first notes that by applying a linear transformation, we may assume, WLOG, that these $d$ shifts are the standard basis for $\R^d$. Then, in order to obtain the events $B$ satisfying the property described above, we choose $\eps=\frac{1}{m}$ for a large enough integer $m$ and proceed analogously to what we described in the second paragraph of this proof.\qedhere
\end{proof}

\end{document}